\def\VER{{{\bf VI}:34.0}}
\documentclass [12pt,a4paper,reqno]{amsart}
\newcommand\Cref[1]{{Corollary~\ref{#1}}}
\newtheorem{thm}{Theorem}[section] 

\newtheorem{defn}[thm]{Definition}

\newtheorem{exmpl}[thm]{Example}

\newtheorem{lem}[thm]{Lemma}
\newtheorem{prop}[thm]{Proposition}

\newtheorem{rem}[thm]{Remark}

\newcommand\Dref[1]{{Definition~\ref{#1}}}

\newcommand\M[1][d]{{\operatorname{M}_{#1}}} 
\def\CI{\mathcal I}
\def\CJ{\mathcal J}

\def\({\left(}
\def\){\right)}

\def\bM{\overline{\mathcal M}}
\def\MG{\mathcal M}

\newcommand{\smat}[4]{{\(\!\!\begin{array}{cc}{#1}\!&\!{#2}\\[-0.1cm]{#3}\!&\!{#4}\end{array}\!\!\)}} 

\def\id{\operatorname{id}}

\def\tr{\operatorname{tr}}

\def\a{{\alpha}}

\def\alphaj{\a_{\operatorname{mat}}}

\def\aqpol{{\operatorname{\alpha}_{\operatorname{pol}}}^{\bar q}}

\def\la{{\lambda}}

\def\sub{{\,\subseteq\,}}
\newcommand{\set}[1]{{\left\{#1\right\}}}

\newcommand\eq[1]{{(\ref{#1})}}
\newcommand{\card}[1]{{\left|{#1}\right|}}

\def\II{{I\!\!\,I}}

\def\N {{\mathbb {N}}}

\DeclareMathOperator{\mychar}{char} %
\long\def\forget#1\forgotten{}
\newcommand\suchthat{{\,:\ \,}}

\newcommand\comp[3][\bullet]{{{#1}_{{\if1#2{}\else{#2}\fi}{\if#3K{}\else{(#3)}\fi}}}} 

\newif\ifXY 
\XYtrue     
\ifXY \usepackage{xy}\fi %
\ifXY \xyoption{matrix}\xyoption{arrow}\xyoption{curve} \fi

\def\Zcd{{Zariski closed}}
\def\Zcr{{Zariski closure}}

\usepackage[usenames]{color}

\begin{document}

\title[Representability of affine algebras  over an arbitrary field  (Ver. \VER)] {Representability of affine
algebras  over  an arbitrary field }

\author{Alexei Belov-Kanel}
\author{Louis Rowen}
\author{Uzi Vishne}

\address{Department of Mathematics, Bar-Ilan University, Ramat-Gan
52900,Israel} %
\email{\{belova, rowen, vishne\}@math.biu.ac.il}

 \subjclass[2010]  {Primary:  16R10 Secondary: 16R30, 17A01, 17B01, 17C05}

\date{\today}


\keywords{Kemer's representability theorem, Specht's question,
hiking, Shirshov's Theorem, polynomial identities, T-ideal, affine
algebra, relatively free, representable}

\thanks{This research was supported  by the
Israel Science Foundation, (grant No. 1178/06).}

\begin{abstract}
In a series of papers culminating in \cite{BRV5}, summarized in
\cite{BRV6}, we used full quivers as tools in describing
PI-varieties of algebras and providing
 a complete proof of Belov's solution of
Specht's problem for affine algebras over an arbitrary Noetherian
ring.  In this paper, utilizing ideas from that work, we give a full
exposition of Belov's theorem~\cite{B2} that relatively free affine
PI-algebras over an arbitrary field are representable.
\end{abstract}

\maketitle
\tableofcontents
\newcommand\LL[2]{{\stackrel{\mbox{#1}}{\mbox{#2}}}}
\newcommand\LLL[3]{\stackrel{\stackrel{\mbox{#1}}{\mbox{#2}}}{\mbox{#3}}}
\newcommand\LLLL[4]
{\stackrel {\stackrel{\mbox{#1}}{\mbox{#2}}}
{\stackrel{\mbox{#3}}{{\mbox{#4}}}} }
\newcommand\AR[1]{{\begin{matrix}#1\end{matrix}}}

\section{Introduction}

 In this paper, utilizing ideas from \cite{BRV5}, we give a full
exposition of Belov's theorem~\cite{B2} that relatively free affine
PI-algebras over an arbitrary field are representable. As
in~\cite{BRV5}, the main tool in utilizing the combinatorics of
polynomials is ``hiking,'' which however is more complicated here
since it involves non-homogeneous polynomials, and is described
below in several stages.

 We work with algebras over a field $F$, with special emphasis to
the possibility that $F$ is finite. A~\textbf{(noncommutative)
polynomial} is an element of the free associative algebra~$F\{ x\}$
on countably many generators. A~\textbf{polynomial identity} (PI) of
an algebra $A$ over $F$ is a noncommutative polynomial  which
vanishes identically for any substitution in~$A$. We use \cite{BR}
as a general reference for~PIs. A~T-\textbf{ideal}  of $F\{ x\}$  is
an ideal $\CI$ of $F\{ x\}$ closed under all algebra endomorphisms
$F\{ x\}\to F\{ x\}$. We write $\id (A)$ for the T-ideal of PIs of
an algebra~$A$.

 Conversely, for any T-ideal $\CI$ of $F\{ x\}$, each element of
$\CI$ is a PI of the algebra  $F\{ x\}/\CI$, and $F\{
x\}/\CI$ is \textbf{relatively free}, in the sense that for
any PI-algebra $A$ with $\id(A) \supseteq \CI,$ and any $a_1,
a_2,\ldots\in A,$ there is a natural homomorphism $F\{ x\}/\CI
\to A$ sending $x_i \mapsto a_i$ for $i = 1,2, \dots .$


\subsection{Representability}$ $

 An
$F$-algebra $A$ is called {\bf representable} if it is embeddable as
an $F$-subalgebra of $\M[n](K)$ for a suitable field $K \supseteq
F$.
 Obviously any representable algebra is PI, but an easy counting argument
of Lewin~\cite{Lew} leads to the existence of non-representable
affine PI-algebras.

Nevertheless, the representability question of relatively free
affine algebras has considerable independent interest, and the
purpose of this paper is to give a full, self-contained proof of the
following result proved by Kemer over any infinite field, and to
elaborate Belov~\cite{B2} (over a finite field):

 \begin{thm}\label{4.66} Every relatively free affine PI-algebra over an arbitrary field is  representable. \end{thm}

Kemer obtained Theorem~\ref{4.66} over infinite fields by means of the following amazing results:

\begin{thm}[{\cite[Theorem~4.66]{BR}, \cite{Kem11}}]\label{fd1}$ $
\begin{enumerate}
\item Every affine PI-algebra over an infinite field (of arbitrary characteristic) is PI-equivalent  to a
finite dimensional (f.d.)~algebra.
\item Every   PI-algebra of characteristic 0 is PI-equivalent  to the Grassmann envelope of a
finite dimensional (f.d.)~algebra.
\end{enumerate}
\end{thm}

In~\cite{BRV1}--\cite{BRV5} we have provided a complete proof for
the affine case of Specht's problem in arbitrary positive
characteristic. (The non-affine case has counterexamples, cf.~\cite{B0,B1}.) Together with Kemer's solution in characteristic 0,
this leads to:

\begin{thm}\label{Spaff}
Any affine PI-algebra over an arbitrary commutative Noetherian ring
satisfies the ACC on T-ideals.
\end{thm}

\begin{rem}\label{Noethind2} For graded associative algebras \cite{AB} and various nonassociative affine algebras of characteristic 0,
the finite basis of T-ideals has been established in the case when
the operator algebra is PI (Iltyakov~\cite{Ilt91,Ilt92} for
alternative and Lie algebras, and Vais and Zelmanov \cite{VZ} for
Jordan algebras) but the representability question remains open for
nonassociative algebras, so representability   presumably is more
difficult. The obstacle is getting started via some analog of
Lewin's theorem \cite{Lew}, which is not yet available.
Belov~\cite{B5} proved representability of alternative or Jordan
algebras satisfying all identities of some finite dimensional
algebra.
\end{rem}

\subsubsection{Plan of the proof of Theorem~\ref{4.66}}$ $

An immediate consequence of Theorem~\ref{fd1}(1) is that every
relatively free affine PI-algebra over an infinite field is
representable, since it can be constructed with generic elements
obtained by adjoining commutative indeterminates to the
f.d.~algebra.
  Kemer deduced from this the finite basis of T-ideals    (the solution of Specht's problem) over an infinite
field, in which he applied combinatorial techniques to representable
algebras.

The approach here for positive characteristic, following \cite{B2},
is the reverse, where one starts with the solution of Specht's
problem and applies Noetherian induction to prove representability
of affine relatively free PI-algebras over arbitrary fields
(including finite fields). (These methods also work in
characteristic 0, but rely on Kemer's solution of Specht's problem
in characteristic 0, which in turn relies on his representability
theorem in characteristic 0.)

\begin{rem}\label{Noethind}  We fix the following notation: We start
with the free affine algebra $F \{ x \} = F\{ x_1, \dots, x_\ell\}$
in $\ell$ indeterminates, and a T-ideal $\CI.$ This gives us the
relatively free algebra
$$A = F \{ x \}/\CI.$$  We say that the T-ideal $\CI$ is
\textbf{representable} if the affine algebra $A $ is  representable.

  The proof of Theorem~\ref{4.66} goes along
the following version of Noetherian induction:

We aim to show that every T-ideal $\CI$ is
   representable.
In view of Lewin's theorem \cite{Lew}, $\CI$ contains a
representable T-ideal $\CI_0$, so we assume that $A_0: = F \{ x
\}/\CI_0$ is representable. In view of Theorem~\ref{Spaff}, we have
a maximal representable T-ideal $\CI_1\supseteq \CI_0$ of $A$
contained in $\CI$, which we aim to show is~$\CI$. Assuming on the
contrary that $\CI_1 \subset \CI,$ we replace $\mathcal{I}_0$ by
$\mathcal{I}_1$ and $A_0$ by $A_0/\CI_1$. This reduces us to the
case where $A_0$ is representable but every nonzero T-ideal of~$A_0$
contained in $\CI$ is not representable. Our goal is to arrive at a
contradiction by finding a representable T-ideal $\CJ \subseteq \CI$
which strictly contains $\CI_1$. We will do this by taking some $f
\in \CI \setminus \CI_1$, i.e., $f \notin \id(A_0)$, and finding
$\CJ$ inside the T-ideal generated by $f$. (This process will
terminate because of the solution to Specht's problem given in
\cite{BRV5}. For this reason, we do not need to introduce parameters
of the induction.)
\end{rem}

The rest of this paper consists of the proof of Theorem~\ref{4.66}
by means of Remark~\ref{Noethind}. The proof  relies on the ideas of
the proof of Kemer's representability theorem given in
\cite{BR,BKR}. Much of this paper is devoted to elaborating the
theory of \cite{BRV2} and \cite{BRV3}, as described in
\cite{BRV5,BRV6}, and there is a considerable overlap with
\cite{B2}.

The proof of Theorem~\ref{Spaff} in \cite{BRV5} is somewhat
different from Kemer's proof. In~\cite{BRV2} we considered the {\bf
full quiver} of a representation of an associative algebra over a
field, and determined properties of full quivers by means of a close
examination of the structure of Zariski closed algebras, studied in
\cite{BRV1}. Then we modified $f$ by means of a ``hiking procedure''
in order to force $f$ to have certain combinatorial properties, and
used this to carve out a T-ideal $\mathcal J$ from inside a given
T-ideal; modding out $\mathcal J$ lowers the quiver in some sense,
and then one obtains Specht's conjecture by induction.

Our approach here is similar, but with some variation. Here we need
not mod out by $\mathcal J$, but need $\mathcal J$ to be
representable. We start the same way, but one of the key steps
fails, and we need a way of getting around it. In both instances,
our techniques rely on the theory of full quivers of a f.d.~algebra
over a field, to be recalled, followed by adjunction of
characteristic coordinates.

We  introduce ``critical'' polynomials (Definition~\ref{doccr})
which enable us to calculate characteristic coefficients using the
combinatoric properties of polynomials, leading to our main tool:

\begin{thm}[Canonization Theorem for Polynomials]\label{hikthm0} Suppose $f(x_1, \dots, x_t) $ is a nonidentity of $A_0$ whose nonzero
evaluation passes through all the blocks of the quiver, via the
dominant branch. Then the T-ideal of $f$ contains a critical
nonidentity.
\end{thm}

The proof of the Canonization Theorem for Polynomials is based on
applying hiking to obtain more and more complicated polynomials
while preserving the two ``Kemer invariants''  of the polynomial
described in \cite{BR,BKR}, which underly the computational study of
T-ideals. The basic operations of hiking, namely multiplying by a
Capelli polynomial or replacing a radical element by a commutator
element of the same form preserves the hypotheses of
\cite[Lemma~6.7.3]{BKR}, so we can measure the dimension of the
semisimple part and the nilpotence index of the radical in terms of
the Kemer  invariants.

The Canonization Theorem for Polynomials will enable us to replace
multiplication by characteristic coefficients in the Shirshov
extension, with multiplication by elements of $A_0$.

\section{Preliminaries}

Let us review some of the techniques we need for the proof.  The
reader can refer to \cite{BRV6} for further details of all of this
material.

\subsection{Linearization and quasi-linearization}$ $

The well-known linearization process of a polynomial can be
described in two stages: First, writing a polynomial $f(x_1, \dots,
x_n)$ as
$$f(0,x_2, \dots, x_n)+ (f(x_1, \dots, x_n)- f(0,x_2, \dots, x_n)),$$ one sees by
iteration that any T-ideal is additively spanned by T-ideals of
 polynomials for which each indeterminate appearing nontrivially
appears in each of its monomials,
 cf.~\cite[Exercise~2.3.7]{Row1}. Then we define the \textbf{linearization process}
 by
 introducing a new indeterminate $x_i'$ and
 passing to $$f(x_1, \dots, x_i +x_i', \dots, x_m) - f(x_1, \dots, x_i , \dots,
 x_m)-f(x_1, \dots, x_i', \dots, x_m).$$
This process, applied repeatedly, yields a multilinear polynomial in
the same T-ideal. In characteristic 0 the linearization process can
be reversed by taking $x_i' = x_i,$ implying that every T-ideal is
generated by multilinear polynomials. But this fails in positive
characteristic, as exemplified by the Boolean identity $x^2 -x$, so
we need an alternative. To handle characteristic $p>0$, Kemer
~\cite{Kem2} took a closer look, which we review from~\cite{BRV6}.

\begin{defn}\label{QL}
A function $f$ is $i$-\textbf{quasi-linear} on $A$ if
$$f(\dots, a_i + a_i', \dots) = f(\dots, a_i , \dots)+f(\dots,
a_i', \dots)$$ for all $a_i, a_i' \in A; $ $f$ is
$A$-\textbf{quasi-linear} if $f$ is $i$-quasi-linear on $A$ for all
$i$.   When $A$ is understood, we just say
\textbf{quasi-linear}.\end{defn}

Suppose $f(x_1, x_2, \dots) \in F\{ x\}$ has degree $d_i$ in $x_i$.
The \textbf{$i$-partial linearization} of~$f $ is
\begin{equation}\label{partlin}\Delta_i f := f(x_1, x_2, \dots, x_{i,1}+\cdots+ x_{i,d_i}, \dots)- \sum
_{j=1}^{d_i} f(x_1, x_2, \dots, x_{i,j}, \dots)\end{equation} where
the substitutions were made in the $i$ component, and
$x_{i,1},\dots,x_{i,d_i}$ are new variables.

 When $\Delta_if (A)= 0$,
then $f$ is $i$-quasi-linear on $A$, so we apply \eqref{partlin} at
most $\deg_i f$ times repeatedly, if necessary, to each $x_i$ in
turn, to obtain a nonzero polynomial that is $A$-quasi-linear in the
T-ideal of $f$.

\begin{prop}[{\cite[Corollary~2.13]{BRV4}}]\label{quasi}
Assume $\mychar F = p > 0$. For any polynomial~$f$ which is not an
identity of $A_0$, the T-ideal generated by $f$ contains a
quasi-linear non-identity  for which the degree in each
indeterminate is a $p$-power.
\end{prop}

\subsection{Full quivers}$ $

In this subsection $A_0$ is a 
representable affine algebra over a field $F$,   i.e., $A_0 \subset M_n(K)$ with $K$
finite or algebraically closed, and we fix this particular
representation. The closure of $A_0$ in $M_n(K)$ with respect to the
Zariski topology \cite[\S~3.1]{BRV6} is PI-equivalent to $A_0$, so
we assume throughout that $A_0$ is Zariski closed. In particular,
when $F$ is infinite then we may assume $F =K$,
cf.~\cite[Remark~3.1]{BRV6}. By Wedderburn's Principal Theorem
\cite[Theorem 2.5.37]{Row1.5}, $A_0 = S \oplus J$ as vector spaces,
where $J$ is the radical of $A_0$ and $S \cong A_0/J$ is a
semisimple subalgebra of $A_0$. Thus $S$ is a direct product of
matrix algebras $R_1 \times \dots \times R_k$, which we want to view
along the diagonal of $M_n(K)$, although perhaps with identification
of coordinates, which are to be described graphically.  By the
Braun-Kemer-Razmyslov theorem, cf.~\cite{Br},   $J$ is nilpotent, so we
take $t = t_{A_0}$ maximal such that $J^t \ne 0.$

We need an explicit description, but which may distinguish Morita
equivalent algebras since matrix algebras of different size are not
PI-equivalent. The {\bf full quiver} of $A_0$ is a directed graph
$\Gamma$, having neither loops, double edges, nor cycles, with the
following information attached to the vertices and edges:

The vertices of the full quiver of $A_0$ correspond to the diagonal
matrix blocks arising in the semisimple part $S$, whereas the arrows
come from the radical~$J$. Every vertex likewise corresponds to a
central idempotent in a corresponding matrix block of $M_n(K)$. 

\begin{itemize}

\item The vertices are ordered, say from
$\bf 1$ to $\bf k$, and an edge always takes a vertex to a vertex of
higher order. There are identifications of vertices, called
\textbf{diagonal gluing}, and identification of edges, called
\textbf{off-diagonal gluing}. Gluing of vertices in full quivers is
identical or \textbf{Frobenius}, as in $\set{\smat{\alpha}{0}{0}{\alpha^q} \suchthat \alpha \in K}$ where $\card{F} = q$ and $K = \overline{F}$.
\item Each vertex is labeled
with a roman numeral ($I$, $\II$ etc.); glued vertices are labeled
with the same roman numeral. A
vertex can be either \textbf{filled} or \textbf{empty}. 

The first vertex listed in a glued matrix block is also given a pair
of subscripts --- the \textbf{matrix degree} $n_{\bf i}$ and the
\textbf{cardinality} of the corresponding field extension of $F$
(which, when finite, is denoted as a power $q^{t_{\bf i}}$ of $q
=\card{F}$).
\item When the base field $F$ is finite, superscripts  indicate the
\textbf{Frobenius twist} between glued vertices, induced by the
Frobenius automorphism $a \mapsto a^q;$ this could identify $a^{q_1}
$ with $a^{q_2}$ for powers $q_1,q_2$ of $q$ (or equivalently $a$
with $a^{q_2/q_1}$ when $q_1 < q_2$); we call this
\textbf{$(q_1,q_2)$-Frobenius gluing}.

\item Off-diagonal gluing (i.e., gluing among the edges) includes Frobenius gluing (which only
exists in nonzero characteristic) and \textbf{proportional gluing}
with an accompanying \textbf{scaling factor}~$\nu$.
\end{itemize}
 Examples are given   in \cite{BRV3}. Now we take some non-identity of $A_0$, say $f(x_1,
\dots, x_m) = \sum g_j (x_1, \dots, x_m)\in \CI$ for
monomials $g_j$. An easy technical condition: Since the full quiver
$\Gamma$ of $A_0$ could be replaced by the subquiver corresponding
to the algebra generated by evaluations of all polynomials in
$\CI$, and then $f$ could be replaced by a sum of polynomials
in~$\CI$, we may assume that $\Gamma$ passes through all
blocks.

The numbers $\dim _F A_0$ and  $t$ are crucial to the description of
quivers, so we want these numbers to be reflected in the polynomial
$f$. This is achieved by means of  Kemer's First
Lemma~(\cite[Proposition~6.5.2]{BKR}) and  Kemer's Second
Lemma~(\cite[Proposition~6.6.31]{BKR}). On the other hand, we need
$f$ to be \textbf{full}
  on the f.d.~algebra $A_0$ in the sense that
    some nonzero
evaluation of $f$  passes through all the blocks of the quiver, via the
dominant branch~$\mathcal B$.     This is
achieved by means of Lemma \cite[Proposition~6.7.3]{BKR}, called the
\textbf{Phoenix property}. Applying these results to $f$ after
hiking (to be described below), we assume throughout that $f$ is
full, and that the conclusion of Kemer's First Lemma and Kemer's
Second Lemma hold.

In view of Proposition~\ref{quasi} we may assume that $f$ is
quasi-linear. When specializing~$x_i$ to~$A_0$, we write the
substitutions $\bar x_i$ as sums of radical and  semisimple
elements; since $f(x_1, \dots, x_m)$ is quasi-linear, we reduce the
substitutions in $S+J$ to their component parts in $S  \cup J$; we
call these substitutions \textbf{pure}. Thus $f$  has a nonzero
specialization where all substitutions $\bar x_i$ are pure. We fix
this specialization and the notation $\bar x_1, \dots, \bar x_m.$
Any other specialization is denoted $\bar x_i'$.

Any pure semisimple substitution $\bar x_i$ is  in $S$ and thus in a
block (or in glued blocks) of some degree $n_i$, which we also call
the \textbf{degree} of $\bar x_i$. A~radical substitution $\bar x_i$
is somewhat more subtle. It is viewed as an edge connecting two
vertices in blocks, say of degrees $n_{i_1}$ and $n_{i_2} $. If
these blocks are not glued, then we call this substitution a
\textbf{bridge} of \textbf{degrees} $n_{i_1}$ and $n_{i_2} $. A
bridge is \textbf{proper} if $n_{i_1} \ne n_{i_2} $. A~proper bridge connecting vertices of degree $n_i \ne n_j$ is an
$\tilde n$-\textbf{bridge} if $n_i$ or $n_j$ is $\tilde n$. But there also
is the possibility that a radical substitution connects two glued
blocks of the same degree $\tilde n$, in which case we call it
$\tilde n$-\textbf{internal}.

%
%
%
%
%

\subsubsection{Review of the three  canonization theorems for quivers}$ $

Since arbitrary gluing is difficult to describe, we need some
``canonization'' theorems to ``improve'' the gluing. The first
theorem shows that we have already specified enough kinds of gluing.

\begin{thm}[{First Canonization Theorem, cf.~\cite[Theorem~6.12]{BRV2}}]\label{main} The \Zcr\ of any representable affine PI-algebra $A_0$ has a representation for
whose full quiver all gluing is proportional Frobenius.
\end{thm}

For the Second Canonization Theorem we grade paths according to the
following rule:

\begin{defn} When $|F| = q < \infty,$ we write $\MG_\infty$ for the multiplicative monoid $\set{1,q,q^2,\dots,\epsilon}$,
where $\epsilon a = \epsilon$ for every $a \in \MG_{\infty}$. (In
other words, $\epsilon$ is the zero element adjoined to the
multiplicative monoid $\langle q \rangle$.)   Let $\bM$ be the
semigroup $\MG_\infty /\!\!\sim$
 where $\sim$ is the equivalence relation obtained by
matching the degrees of glued variables: When two vertices have a
$(q_1,q_2)$-Frobenius twist, we identify $1$ with  $q^k =
\frac{q_1}{q_2}$ in the respective matrix blocks, and use $\bM$ to
grade the paths.
\end{defn}

\begin{defn}
A~full quiver is \textbf{basic} if it has a unique initial vertex
$r$ and unique terminal vertex $s$, and all of its gluing above the
diagonal is proportional Frobenius. A~basic full quiver  $\Gamma$ is
\textbf{canonical} if
any two paths from the vertex $r$ to the vertex $s$ have the same
grade.
\end{defn}

(Our notion of basic quiver has nothing to do with the notion of
basic algebra in representation theory.)

\begin{thm}[{Second Canonization Theorem, cf.~\cite[Theorem~3.7]{BRV3}}]\label{Can2}
Any relatively free algebra is a subdirect product of algebras whose
full quivers are basic.

Any basic full quiver $\Gamma$ of a representable relatively free
algebra can be modified (via a change of base) to a canonical full
quiver  of an isomorphic algebra (i.e., relatively free algebra of
the same variety).
\end{thm}

In view of this result, we may reduce to the case that the full
quiver of our polynomial~$f$ is basic.

 The Third Canonization
Theorem \cite[Theorem~3.12]{BRV3} describes what happens when one
mods out a ``nice'' T-ideal, so is not relevant, since all we need
is to find a representable T-ideal, which we do later by another
method.

\section{The Canonization Theorem for Polynomials}$ $

We have two languages: quivers and their representations on one hand, versus the
combinatorial language of identities on the other hand. First we consider the
geometrical aspect. A \textbf{branch} is a path  $\mathcal B$ in the
quiver.
 The
 {\bf length} of $\mathcal B$ is its number of arrows,
excluding loops, which equals its number of vertices (say $k$) minus $1$.
Thus, a typical branch has vertices of various matrix degree $n_j$,
$j = 1, 2, \dots,k$. We call $(n_1, \dots, n_k)$ the \textbf{degree
vector} \cite[Definition~2.32]{BRV5} of the branch $\mathcal B$. The
\textbf{descending degree vector} is obtained by ordering the entries
of the degree vector to put them in descending order
lexicographically (according to the largest $n_j$ which appears in
the distinct glued matrix blocks, excluding repetitions, taking the
multiplicity into account in the case of Frobenius gluing). We write
the descending degree vector as $(\pi(\bold n)_1, \dots, \pi(\bold n)_k)$. Thus,
$\pi(\bold n)_1= \max\set{n_1,\dots,n_k}$.
If $\mathcal B$ appears in a
nonzero specialization of a monomial of~$f$, we call $\mathcal B$ a
\textbf{branch of} $f$.

We denote the largest $n_j$ appearing in the quiver as $\tilde n$.

\begin{defn}\label{def21}
A  branch  $\mathcal B$ is  \textbf{dominant} if it has the maximal
possible number of $\tilde n$-bridges, has maximal possible length
$k$ with regard to this property, has the maximal possible number of
vertices of $\tilde n$-bridges among these in the lexicographic
order, 
and then we continue down
the line to $\tilde n -1$ etc. The {\textbf{depth}} of a dominant
branch $\mathcal B$ is the number of times $\tilde{n}$ appears in
its degree vector.
%
\end{defn}

Our goal is somehow to force  every  nonzero evaluation of $f$ into
a dominant branch by considering each degree from $\tilde n$ down in
turn. Throughout, $c_m$ here denotes the Capelli polynomial in
$2m^2$ indeterminates (denoted $c_{2m^2}$ in \cite{BR}), which is
alternating in $m$ indeterminates and an identity of $M_{m-1}(F)$
for any field $F$; $h_{m,i}(y)$ denotes a multilinear central
polynomial $h_{m,i}(y_{i_1}, \dots, y_{i_m'})$ for $\M[m](F)$, in
specific  indeterminates $y_{i_1}, \dots, y_{i_m}$  which are all distinct. %
Evaluating  $h_ {m,i} $ on semisimple matrices of degree $<m$ is~$0$.
We put
$$h _m= h_{m,1} h_{m,2}\cdots h_{m,t+1},$$  the product of $t+1$ copies of distinct Capelli polynomials of the same degree $2m^2$, and call the $h_{m,j}$ the respective \textbf{components} of~$h$. We focus first on semisimple substitutions having matrix degree
$\tilde n,$ and put $h = h_{\tilde n}.$

\begin{lem}\label{expans10} Any nonzero specialization of $h$ has a component of solely semisimple substitutions (all of the same
degree).
\end{lem}
\begin{proof} Otherwise every component has a radical substitution,
so we have a product of $t+1$ radical elements, which is 0 by
definition of $t$. \end{proof}

Viewing a substitution of $x_i$ as corresponding to an edge in the quiver,
we have two degrees, one for each vertex.

\begin{defn} An \textbf{$m$-right} substitution of $x_i$ is one having
degree including $m$. An \textbf{$m$-wrong} substitution
$\overline{x_i}$ of $x_i$ is one both of whose  degrees differ from
$m $, where $m$ appears as a degree  of the substitution $\overline{x_i}$.

 We write \textbf{right} (resp.~ \textbf{wrong})
for $\tilde n$-right (resp.~$\tilde n$-wrong).
\end{defn}

In view of Lemma~\ref{expans10}, a wrong substitution would lead to
$h$ having a component of semisimple substitutions in a matrix block
of the wrong degree.

One delicate point:  An internal radical bridge say from one matrix
block of degree $m$ to a different  matrix block of degree $m$ is
technically ``right'' according to this definition, but must be
dealt with.
%
%
%

\begin{rem} Suppose $f(x_1, \dots, x_\ell) $ is a full nonidentity of $A_0$ whose nonzero
evaluation passes through all the blocks of the quiver, via the
dominant branch~$\mathcal B$ say of degrees $m_1, \dots, m_k$ having
some number $k$ of bridges, and $k'$ internal radical substitutions.
By Theorem~\ref{Can2}, any wrong nonzero substitution may be assumed
to have $k$ bridges since otherwise we apply induction  to the
number of semisimple components in the full quiver. On the other
hand, we can take the nonzero evaluation with $k'$ maximal, so then
any wrong substitution has at most $k'$ internal radical
substitutions.
\end{rem}

We work with a dominant branch $\mathcal B$ of $\Gamma$ in $f$. Our
objective is to modify $f$ to a nonidentity containing a Capelli
component which enables us to use combinatorial methods to calculate
characteristic coefficients in a Shirshov extension, with
multiplication by elements of $A_0$. Here is one of our main
  results, enabling us to correspond quivers with properties of
  polynomials, and which leads directly to the representability theorem.

\begin{thm}[Canonization Theorem for Polynomials]\label{hikthm} Suppose $f(x_1, \dots, x_\ell) $ is a full nonidentity of $A_0$.
Then the T-ideal of $f$ contains   a critical
nonidentity.
\end{thm}

\section{The proof of
the Canonization Theorem for Polynomials}$ $

 The proof of
the Canonization Theorem for Polynomials is
done in several stages:

\begin{enumerate}
\item Eliminate unwanted semisimple substitutions.
\item Make sure that that the substitutions are in the ``correct''
semisimple components.
\item Provide a molecule inside the polynomial where we
can compute the action of characteristic coefficients.
 \end{enumerate}

\subsection{Unmixed case}\label{nonmixed}$ $

First, following Kemer, we dispose of the following easy case.
 We say a
substitution is \textbf{unmixed} if it does not involve bridges,
i.e., all substitutions are in a single Peirce component. Here we
need only multiply by a Capelli polynomial of the matrix degree, and
then proceed directly to the method of \S\ref{trab}.

This aspect is crucial to our proof, since substitutions alone are
not sufficient to take care of examples such as the non-finitely
generated T-space of \cite{Shch00} (generated by
$\{[x_1,x_2]x_1^{p^k-1}x_2^{p^k-1}, \ k \in \N\}$ in the Grassmann
algebra with two generators; also see~\cite{G,G2}).

%

\subsection{The mixed case: the hiking procedure}$ $

To complete the proof of
the Canonization Theorem for Polynomials,
we must turn to the mixed case.
The main feature in the proof of the Canonization Theorem for
Polynomials is hiking. The notion of hiking passes from branches of
quivers to combinatorics of nonidentities, showing how to modify a
non-identity of a T-ideal $\CI$ to  another non-identity in
$\CI$ whose algebraic operations leave us in the same quiver.

 In our combinatorics we need to cope with the danger that our
 substitutions are  wrong,  or the base field of the semisimple component
 is of the
wrong size. To prevent this, we make  substitutions of multilinear
polynomials for indeterminates inside $f$, called \textbf{hiking},
which force the evaluations to become 0 in such situations. In other
words, hiking replaces $f$ by a more complicated polynomial in its
T-ideal,
 which yields a  zero valuation when we start with a wrong substitution
 in the original indeterminates of~$f$.

We have three kinds of variables: \begin{itemize} \item Core
variables, used for exclusive absorption inside the radical (such as
variables which appear in commutators with central polynomial),\item
variables used for hiking, \item  variables inside Capelli
polynomials used for computing the actions of characteristic
coefficients.
 \end{itemize}

\begin{exmpl} An easy example of the underlying principle:
If $k = 2$ with $n_1> n_2$, then the quiver $\Gamma$ consists of two
blocks and an arrow connecting them, so we  replace a variable $y$
of $f$ with a radical substitution by $h_{n_1,1}[h_{n_1,2},z] y
h_{n_2}$. The corresponding specialization remains in the radical.
Then we are ready to utilize the techniques given below in
\S\ref{trab} to compute characteristic coefficients, bypassing the
complications of hiking.
\end{exmpl}

Suppose we have the polynomial $f$, with a radical evaluation. We
replace it and have a hiked polynomial. If $\hat g$ belongs to the
T-ideal generated by $g$, then one of the variables in $h$ must have
a radical evaluation. After making the substitution we get a new
polynomial $g'$ of the same form as $g$. This is like the Phoenix
property described in characteristic 0. But in general we need a
rather intricate analysis.


\begin{defn}\label{mol} Given a polynomial $f(x_1, \dots,x_\ell)$  and another polynomial $g$, we write $f_{x_i \mapsto g}$ to
denote that $g$ is substituted for $x_i$. We say that $f$ is
\textbf{hiked} to $\tilde f : = f_{x_i \mapsto g}$ at $x_i$ if $g$
is linear in $x_i$.

We call the replacement $g$ of $x_i$ a \textbf{molecule} of the
hiked polynomial. A \textbf{complex molecule} is the product of
molecules.
 \end{defn}

 \begin{defn}\label{doccr} A polynomial is \textbf{docked} (of \textbf{length} $d$)
if it can be written in the form
 $$\sum _u g_{u,1}h_{\tilde n}(y) g_{u,2} h_{\tilde n}(y)\cdots g_{u,d} h_{\tilde n}(y)g_{u,d+1}$$
 for suitable polynomials $g_{u,i}$ (perhaps constant) in which the
 $y$ indeterminates do not occur. (In other words the
 $y$ indeterminates  occur only in the $h_{\tilde n}(y)$.) Unfortunately, if the $x_i$ repeat
 then the molecules repeat, and thus the variables $y$ repeat.

 A docked polynomial $f(x_1, \dots, x_t; y,y',y''; z,z')$ is \textbf{critical} if any nonzero substitution
of the $y_i$ is right.
\end{defn}
Thus the docks are attached to molecules. If $f$ is hiked to various
polynomials $f_j$ we also say it can be hiked to $\sum f_j$.

(Likewise for other indeterminates that appear once the hiking is
initiated.)


\begin{rem} First suppose that the depth $u = k,$ i.e., all $n_j = \tilde n$,
and there are no nonzero external radical substitutions. In other
words, the only nonzero substitutions involve specializing all
the $x_i$ to semisimple elements in blocks of degree $\tilde n$.
Then we simply replace $f$ by $hf$, which trivially is docked, and
the theorem is proved. So in the continuation, we assume that $u<k$,
which means there is some nonzero substitution $f(\overline{x_1},
\dots, \overline{x_k})$ in our dominant branch $\mathcal B$, for
which some $\overline{x_i}$ is an $\tilde n$-bridge. We fix
$\overline{x_1}, \dots, \overline{x_k}$ in what follows, and call it
our \textbf{fundamental substitution}, with bridge at this $i$.

\end{rem}

%
%

We prove a more technical version of the Canonization Theorem, to handle the mixed case.

\begin{thm}[Hiking Theorem for Polynomials]\label{hikthm7} Suppose $f(x_1, \dots, x_\ell) $ is a  full nonidentity of $A_0$, possibly with mixed or pure
substitutions. Then   $f$ can be hiked to   a critical nonidentity
in which all of the substitutions of the $x_i$ are right.
\end{thm}

\section{Details of hiking}\label{hik}$ $

%
%
%

%

 The
proof of Theorem~\ref{hikthm7} is through a succession of hiking
steps in order both to eliminate ``wrong'' substitutions and then
docking, i.e., insert $h_{\tilde n}$  into the
polynomial.  The latter is achieved by replacing $z_i$ by $h_{\tilde
n} z_i $ and $z_i'$ by $z_i'h_{\tilde n} $; i.e., we pass to $f_{z_i
\mapsto h_{\tilde n} z_i, z_i' \mapsto z_i' h_{\tilde n} } .$

 The hiking procedure   requires
three different stages.

\subsection{Preliminary hiking}\label{firstst1}$ $

Our initial use of hiking is to resolve some technical issues.
First, we want to eliminate the effect of $(q_1,q_2)$-Frobenius
gluing for $q_1 \ne q_2$, since it can complicate docking. Toward
this end, we substitute $z_{i'} c_{n_j}(y)^{q_1/q_2} $ for $z_{i'}$,
for each instance of Frobenius gluing. It makes the Frobenius gluing
identical on $f$.

%


 We also need the base fields of the components all to be the same.
When $\mathcal B'$ is another branch with the same degree vector,
and the corresponding base fields for the $i$-th vertex of $\mathcal
B$ and $\mathcal B'$ are $n_i$ and $n_i'$ respectively, we take $t_i
= q^{n'_i}$ and replace $x_i$ by $(c_{n_i}^{t_i}-c_{n_i})x_i.$ This
cuts off the specializations to matrices over finite fields of the
wrong order.

\subsection{First stage of hiking}\label{firstst}$ $

We have a quasi-linear nonidentity $f$ of a Zariski closed algebra
$A_0$ which has a fundamental substitution in some   branch
$\mathcal B$, where  $\overline{x_{i_1}}$ in $A_0$  is an $\tilde
n$-bridge, corresponding to an edge in the full quiver whose initial
vertex is labeled by $(n_{\ell}, t_{i_1})$  and whose terminal
vertex is labeled by $(n_{i_1+1}, t_{i_1+1})$ where $\tilde  n =
\max \{ n_{i_1}, n_{i_1+1}\}.$ We replace $x_{i_1}$ by $ c_{
n_{i_1} }z_{i_1} [x_{i_1}, h_ {\tilde  n-1} ]
 z_{i_1+1} c_{ n_{i_1+1}  } $,  (where as always the~$c_{n_{i_1}} $ involve new indeterminates in
$v$), and $z_{i_1},z_{i_1+1}$ also are new indeterminates which we
call ``docking indeterminates''); this yields a quasi-linear polynomial in
which any substitution of $x_{i_1}$ into a diagonal block of degree
$<n_1$ or a bridge which is not an $n$-bridge is 0.   For each
semisimple substitution $\overline{x_{i}}$ in a block of degree
$n_i$,  taking $[x_{i_1}, h_ {n_{i_1}}]$ yields 0. This removes all
semisimple component substitutions in $h$ of such $x_i$ whose degree
is too ``small,'' i.e., less than $n_i$. For the time being, we
could still have radical substitutions,
 but first stage hiking does prepare
for their elimination in the second stage.

 The number of
extra $\tilde n$-bridges in a specialization of $ c_{ n_{i_1}
}(v)z_{i_1} [x_{i_1}, h_ {n_{i_1}'}(v)]
 z_{i_1+1} c_{ n_{i_1+1}  }(v) $ is called its (first stage) \textbf{bridge contribution}.
  (In other words, one takes the total number of
 bridges, and subtracts 1 if $\overline {x_i}$ is an $\tilde n$-bridge.) The term $[x_{i_1}, h_ {n_{i_1}'}(v)]$ is called the \textbf{core} of the bridge
 contribution.

\begin{lem}\label{expans15} Any nonzero specialization of $h$ is either $\tilde n$-semisimple, or its bridge
 contribution   is positive.
\end{lem}
\begin{proof} By definition, if the  bridge
 contribution   is 0 then every substitution has to be semisimple or a  $(j,j)$-bridge for some $j$.
 If  $\tilde n$ does not appear then the graph would
have $t$ such bridges.
\end{proof}

\begin{lem}\label{expans16} After the first stage of hiking, a
wrong specialization of an $\tilde n$-semisimple element  cannot be
$m$-semisimple for $m<\tilde n$ unless its bridge contribution is at
least 2.
\end{lem}
\begin{proof} When evaluating $h_{\tilde n}$ on semisimple elements
of degree $m$ we get 0 unless we pass away from the $m$-semisimple
component, which requires two bridges.
\end{proof}

\begin{lem}\label{expans17} After the first stage of hiking, a
wrong specialization of an $\tilde n$-semisimple element is either
$\tilde n$-semisimple or its bridge contribution is at least 1.
\end{lem}
\begin{proof} When evaluating $h_{\tilde n}$ on semisimple elements
of degree $m$ we get 0 unless we pass away from the $m$-semisimple
component, which requires two bridges.
\end{proof}

The first stage of hiking does not instantly zero out bridges
$\overline{x_{i}}$, but does prepare for their elimination in the
second stage.

 Appending the Capelli polynomials also  sets the stage for eliminating other unwanted
substitutions in the second stage.

 After repeated applications of first stage hiking, we  wind up with a new polynomial $f(x_1, \dots, x_\ell;v;z)
$ where we still have our original indeterminates $x_i$ but have
adjoined new indeterminates from $v$ and $z$.


\subsection{Second stage of hiking}$ $

\begin{exmpl} To introduce the underlying principle, here is a slightly more
complicated example. Consider the quiver of three arrows, from
degree 2 to degree 1, degree 1 to degree 1, and finally from degree
1 to degree 1.

 First we multiply on the left by $c_4[h_{1,1},z_1]z_2.$  The second
substitution could have an unwanted position inside the first matrix
block of degree~2, since $c_4[h_{1,1},z_1]z_2$ can be evaluated in
the larger component. We take $f_{x_1 \mapsto c_{2}y'y   x_1} -f_{
x_1 \mapsto x_1 c_{2} y'y} ,$ i.e., we multiply by a central
polynomial $h_2$ on the left and subtract it from a parallel
evaluation of $h_2$ on the right. The unwanted substitution then
cancels out with the other substitution and leaves~0.
\end{exmpl}

 In the second stage of
hiking, in the blended case, we arrange for all nonzero
substitutions to be pure radical.
%
%
%
%

Suppose $f(x_, \dots, x_\ell; y;z;z')$ is already hiked after the
first stage. Suppose in the branch $\mathcal B$ the indeterminate
$z_i$ occurs of degree $d_i$ and the indeterminate $z_{u+1}$ occurs
of degree $d'$, where $1 \le j \le u.$

\begin{prop}\label{expans0} There are three cases to consider:

\begin{enumerate}
\item There is a string $\overline{x_{i-1}}  \overline{x_{i}} \cdots  \overline{x_j} \overline{x_{j+1}}$
where $\overline{x_{i}}, \cdots, \overline{x_j}$ are all semisimple
of the same degree $x_{n_j}$ whereas $\overline{x_{i-1}},
\overline{x_{i}},\ \overline{x_j} \overline{x_{j+1}}$ are both
$\tilde n$-bridges.

We take the polynomial
\begin{equation}\label{OK} f_{z_i \mapsto h_{\tilde n}(y')^{d_i}  z_i}
-f_{ z'_{u+1} \mapsto z'_{u+1} h_{n_u}(y')^{t_i} } ,\end{equation}
where the branch $\mathcal B$ has depth $u$ and $t_i$ designates the
maximal degree of $x_i$ in a monomial of $\mathcal B$, where $y'$ is
a fresh new set of indeterminates

\item There is a string $\overline{x_{1}}  \overline{x_{i}} \cdots \overline{x_j} \overline{x_{j+1}}$
where $\overline{x_{1}} \cdots, \overline{x_j}$ are all semisimple
of the same degree $x_{n_1}$ whereas $\overline{x_{j}}
\overline{x_{j+1}}$ is an $\tilde n$-bridge. We take the polynomial
\begin{equation}\label{OK2} f_{z_1 \mapsto h_{\tilde n}(y')^{d_1}  z_1}
 .\end{equation}

\item There is a string $  \overline{x_{i-1}}   \overline{x_{i}} \cdots  \overline{x_{k}} \overline{x_k}$
where $\overline{x_{i}}, \cdots, \overline{x_{k-1}}$ are all
semisimple of the same degree $x_{n_1}$ whereas $
\overline{x_{i-1}}\overline{x_{i}}$ is an $\tilde n$-bridge. We take
the polynomial
\begin{equation}\label{OK3} f_{z_1 \mapsto h_{\tilde n}(y')^{d_1}  z_1}
 .\end{equation}
 \end{enumerate}

 This hiking zeroes out semisimple evaluations of highest degree
($\tilde n$), but not a radical evaluation at the $u$ block.
\end{prop}

\begin{proof} (Note that (1) is the usual case, but we also need (2)
and (3) to handle terms lying at the ends of the polynomial.) The
expression \eqref{OK} yields zero on a semisimple substitution, but
not on a radical substitution, since exactly one of the two summands
of \eqref{OK} would be 0.
\end{proof}

%
%
%


%

\begin{lem}\label{expans1} The second stage of hiking forces any
nonzero specialization of an $\tilde n$-bridge also to be a $\tilde
n$-bridge.
\end{lem}
\begin{proof} In order to provide a nonzero value, at least one of its
vertices must be of degree~$\tilde n$. But if both were  $\tilde n$
the evaluation would be  $ 0,$ by
Lemmas~\ref{expans15}--\ref{expans17} and Remark~\ref{expans0}. Thus
we get an $\tilde n$-bridge.
\end{proof}

\begin{lem}\label{expans2} After the first and second stages of
hiking, the positions of semisimple substitutions of degree~$\tilde
n$ are fixed; in other words, semisimple substitutions of
degree~$\tilde n$ are $\tilde n$-right.
\end{lem}
\begin{proof} Lemma~\ref{expans1} ``uses up'' all the places for $\tilde n$-bridges,
since more $\tilde n$-bridges would yield a substitution
contradicting the maximality of the number of $\tilde n$-bridges in $\mathcal B$.
If $\mathcal B$ has no semisimple substitutions of degree $\tilde n$
then there is no room for any semisimple substitutions of degree
$\tilde n$, and we are done.

But if  $\mathcal B$ has a semisimple substitution  of degree
~$\tilde n$, that substitution must border an $\tilde n$-bridge,
fixing the order of the pair of indices in the $\tilde n$-bridge,
and thus fixing the positions of all the gaps of index $\tilde  n$
between $\tilde n$-bridges, so we are done.
\end{proof}

\begin{rem}\label{fin} Although this is taken care of in the proof, we can remove finite components
simply by substituting $x_i^m - x_i^\ell$ for $x_i$, for suitable
$\ell,m$.\end{rem}

 {\bf Proof of
Theorem~\ref{hikthm7}.} Just iterate the hiking procedure down from
$\tilde n$.  $\square$

\bigskip

 {\bf Proof of Theorem~\ref{hikthm}.} One obtains the dock
 by replacing $f$ by
$$f_{z_u \mapsto c_{\tilde n}(y')^{t_u} z_u, \ z'_{u+1} \mapsto
z'_{u+1} c_{n_u}(y')^{t_1}}.$$ \hfill $\square$

%
%

\begin{exmpl} Let us run  through the hiking procedure, taking $$A_0
=  \left\{\left(\begin{array}{ccccc}
* & * & * & * & *\\ %
* & * & * & * & *\\ %
* & * & * & * & *\\ %
0 & 0 & 0 & * & * \\
0 & 0 & 0 & * & *
\end{array}\right)\right\}.$$ We have the full quiver $$I_1 \to \II_2,$$ and take the
nonidentity $f = x_1 [x_2, x_3]x_4 + x_4[x_2, x_3]x_1^2.$ We have
nonzero specializations with $\overline{x_1}$ in the first matrix
component, $\overline{x_2}$ an external radical specialization, and
$\overline{x_3}$ in the second matrix component, which we denote as
$\mathcal C = \M[2](K)$, but also we have a nonzero specialization
of all variables into $\mathcal C$. To avoid this situation, we
replace $f$ by
$$f(c_3(y)zx_1z',x_2,x_3,x_4) =  c_3(y) zx_1z' [x_2, x_3]x_4+
 x_4 [x_2,  x_3]^2 c_3(y)  zx_1 z' c_3(y) zx_1 z' .$$   Now any specialization into $\mathcal C $ becomes 0, so
we have eliminated some ``wrong'' specializations. For stage 2 we
take $$\begin{aligned} \tilde f( x;y;y';  z;z')& :  =
f(c_3(y)c_3(y')^2 zx_1z',x_2,x_3,x_4) - f(c_3(y)
zx_1z'c_3(y'),x_2,x_3,x_4)  \\  = (c_3(y)& c_3(y')^2 zx_1z' [x_2,
x_3]x_4+
 x_4 [x_2,  x_3]^2 c_3(y) c_3(y')^2 zx_1 z' c_3(y) zx_1 z')\\ & - (c_3(y)  zx_1z' c_3(y') [x_2,
x_3]x_4+
 x_4 [x_2,  x_3]^2 c_3(y)  zx_1 z' c_3(y') zx_1 z'c_3(y')) ,\end{aligned}$$
 where we see the specialization of highest degree to the first matrix component has
 been eliminated. We can eliminate the nonzero specializations of
 $h(y'')_3$ of degree 1 by taking
 $\tilde f( x;y;y';c_3(y'') z;z')- \tilde f( x;y;y'; z;z'c_3(y''))$
 which leaves us only with  a radical specialization and a
 critical polynomial with a single
 dock $c_3(y'')$.

 Note how quickly the polynomial becomes complicated even though we
 have hiked only one indeterminate.
\end{exmpl}

\begin{rem} Other examples of hiking are given in \cite{BRV5}. The main difference between the hiking procedure of this
paper and that of stage 3 hiking of \cite{BRV5} is in the treatment
of the Frobenius. Stage 4 hiking is analogous.
\end{rem}

\section{Characteristic coefficient-absorbing polynomials inside
T-ideals}\label{trab}$ $

 We have already pinpointed the
  $x_i$ that must have substitutions into semisimple blocks of
degree $\tilde n$, in order to utilize the well-understood
properties of semisimple matrices (especially the coefficients of
their characteristic polynomials, which we call
\textbf{characteristic coefficients}). We follow the discussion of
coefficient-absorbing polynomials from \cite[Theorem~4.26]{BRV5} and
\cite[\S 6.3]{BRV6}, although we can skip much of it because we
already have a docked polynomial.

%
%

Any matrix $a \in \M[n](K)$ can be viewed either as a linear
transformation on the $n$-dimensional space $V = K^{(n)}$, and thus
having Hamilton-Cayley polynomial $f_a$ of degree~$n$, or (via left
multiplication) as a linear transformation $\tilde a$ on the
$n^2$-dimensional space $\tilde V = \M[n](K)$ with Hamilton-Cayley
polynomial $f_{\tilde a}$ of degree $n^2$. The matrix $\tilde a$ can
be identified with the matrix $$a \otimes I \in \M[n](K) \otimes
\M[n](K) \cong \M[n^2](K),$$ so its eigenvalues have the form $\beta
\otimes 1 = \beta$ for each eigenvalue $\beta$ of $a$. From this, we
conclude:

\begin{prop}[{\cite[Proposition~2.4]{BRV3}}]\label{obv2}
Suppose $a \in \M[n](F)$. Then the characteristic coefficients of
$a$ are integral over the $F$-algebra $\hat{C}$ generated by the
characteristic coefficients of $\tilde a$.
\end{prop}
\begin{proof}
The integral closure of $\hat{C}$ contains all the eigenvalues of
$\tilde a,$ which are the eigenvalues of $a,$ so the characteristic
coefficients of $\tilde a$ also belong to the integral closure.
\end{proof}

Next we use hiking also to force the characteristic
coefficients of the matrices to commute with each other.
%
%
%
Using Theorem~\ref{hikthm}, we  work with quasi-linear polynomials
and pinpoint semisimple substitutions of degree $\tilde n$, in order
to utilize the well-understood properties of semisimple matrices
(especially the characteristic coefficients).


Having obtained semisimple substitutions of degree $ \tilde n$, we
have two ways of obtaining intrinsically the coefficients of the
characteristic polynomial
$$g_a = \la^n + \sum_{k=1}^{n-1} (-1)^k \a_k (a) \la ^{n-k}$$ of
a matrix $a$. Fixing $k$, we write $\a_k$ for $\a_k(a),$ which we
call the $k$-\textbf{characteristic coefficient} of $a$. We want to
extract these characteristic coefficients, by means of polynomials.

\begin{defn}\label{trmat0}
In any matrix ring $\M[n](W)$, we define
\begin{equation}\label{trmat}
\alphaj(a) : = \sum_{j=1}^n \sum e_{j,i_1} a e_{i_2,i_2}a \cdots a
e_{i_{k}i_k} a e_{i_1,j},\end{equation} the inner sum taken over all
index vectors of length $k$.
\end{defn}

We can also define the characteristic coefficients
 via polynomials.

\begin{defn}\label{absorp} Given a quasi-linear polynomial $f(x;y)$ in indeterminates
labeled $x_i,y_i$, we say $f$ is {\bf characteristic
coefficient-absorbing} with respect to a full quiver~$\Gamma$ if $f(
A_0(\Gamma))^+$
 absorbs multiplication by any characteristic coefficient of any
 element in each docked (diagonal)
 matrix block of a molecule of $ A_0(\Gamma)$.
 \end{defn}

\begin{lem}[as in {\cite[Lemma~3.6]{BRV4}}]\label{q1}
Write the polynomial $f$ of Theorem~\ref{hikthm} as a sum of
homogeneous components $\sum f_j$. Each $f_j$ is characteristic
coefficient absorbing in the blocks of degree $\tilde n$.
\end{lem}
\begin{proof}
The proof can be formulated in the language of \cite[Theorem~J,
Equation~1.19, page~27]{BR} (with the same proof), as follows,
writing $T_{a,j}$ for the transformation given by left
multiplication by $a$:
\begin{equation}\label{traceab0}
\alpha_k f(a_1, \dots, a_t, r_1, \dots, r_m) = \sum f(T_a^{k_1}a_1,
\dots, T_a^{k_t}a_t, r_1, \dots, r_m),\end{equation} summed over all
vectors $(k_1, \dots, k_t)$ with each $k_i \in \{ 0, 1\}$ and $k_1 +
\dots + k_t = k,$ where $\alpha_k $ is the $k$-th characteristic
coefficient of a linear transformation $T_a: V \to V.$
\end{proof}


%

Since the purpose of Lemma~\ref{q1} was to obtain the conclusion
\eqref{traceab0}, we merely assume~\eqref{traceab0}.
\begin{lem}\label{q2} For any  homogeneous polynomial $f(x_1, x_2,
\dots)$ quasi-linear in $x_1$ with respect to a matrix algebra
$\M[n](F)$, satisfying \eqref{traceab0}, there is a polynomial $\hat
f $ in the T-ideal generated by $f$ which is characteristic
coefficient absorbing.
\end{lem}
\begin{proof} Take the polynomial of Lemma~\ref{q1}.
\end{proof}

\begin{rem}\label{CHid}
Notation as in \eq{traceab0}, the Cayley-Hamilton identity for $n
\times n$ matrices is
$$\begin{aligned} 0 & = \sum_{k=0}^{n} (-1)^k
\alpha_k f(a_1, \dots, a_t, r_1, \dots, r_m)  \la ^{n-k}   \\ & =
 \sum_{k=0}^{n} (-1)^k \sum_{k_1+\dots+k_t= k} f(T_a^{k_1}a_1, \dots, T_a^{k_t}a_t, r_1, \dots,
r_m) \la ^{n-k},\end{aligned}$$ which is thus an identity in the
T-ideal generated by $f$.
\end{rem}

\begin{defn}\label{HCind}
We call the identity $$\sum_{k=0}^{n} (-1)^k \sum_{k_1+\dots+k_t= k}
f(T_a^{k_1}a_1, \dots, T_a^{k_t}a_t, r_1, \dots, r_m) \la ^{n-k}$$
obtained in Remark~\ref{CHid}, the {\bf Hamilton-Cayley identity
induced by $f$}.
\end{defn}

\begin{defn}\label{6.7}
Fixing $0 \le k <n,$ we denote this implicit definition in
Lemma~\ref{q2} of~$\alpha_k ,$ the $k$-th characteristic coefficient
of $a$, as $\aqpol(a)$.
\end{defn}

If the vertex corresponding to $r$ has matrix degree $n_i$, taking
an $n_i \times n_i$ matrix $w$, we define ${\aqpol}_u(w)$ as in the
action of Definition~\ref{6.7} and then the left action
\begin{equation}\label{mtr1}
a_{u,v} \mapsto {\aqpol}_u (w) a_{u,v}.
\end{equation}
Likewise, for an $n_j \times n_j$ matrix $w$ we define the right
action
\begin{equation}\label{mtr2} a_{u,v} \mapsto a_{u,v} {\aqpol}_v (w).
\end{equation}
(However, we only need the action when the vertex is non-empty; we
forego the action for empty vertices.)

\begin{rem}[For $f$ homogeneous.]\label{modhike1}
Take $\hat f$ of Lemma~\ref{q2}, and one more indeterminate $y''$.
There is a Capelli polynomial $\tilde c_{n_i^2}(y'')$ and $p$-power
$\bar q$ such that
\begin{equation}\label{CC=CC}
\tilde c_{n_i^2}( \a_k y'') x_i c_{{n'_i}^2}(y'') = \a_k ^{\bar
q}(y_1)c_{n_i^2}(y'') x_i c_{{n'_i}^2}(y'')\end{equation} on any
diagonal block. Since characteristic coefficients commute on any
diagonal block, we see from this that
\begin{equation}\label{hikemore}\tilde c_{n_i^2}(y'') x_i
c_{{n'_i}^2(y'')}\tilde c_{n_i^2}(z) x_i c_{{n'_i}^2(z)} - \tilde
c_{n_i^2}(z) x_i c_{{n'_i}^2(z)} \tilde c_{n_i^2}(y'') x_i
c_{{n'_i}^2(y'')}\end{equation} vanishes identically on any diagonal
block, where $z = \a_k y''$. One concludes from this that
substituting~\eqref{hikemore} for $x_i$ would hike $\hat f$ one step
further. But there are only finitely many ways of performing this
hiking procedure. Thus, after a finite number of hikes, we arrive at
a polynomial in which we have complete control of the substitutions,
and the characteristic coefficients defined via polynomials commute.
\end{rem}

\section{Resolving ambiguities for nonhomogeneous
polynomials}

When $f$ is homogeneous, we can skip  this section.
 The non-homogeneous case is more delicate.
Since we may be  in nonzero characteristic,   in the main situation
our quasi-linear hiked polynomials are not homogeneous. In
\S\ref{trab} we obtained actions on each monomial component
separately, but we need to provide a uniform action  on each of
these components.

\subsection{Removing ambiguity of matrix degree for nonhomogeneous
polynomials}$ $

First we want to make sure that we are working in the same matrix
degree for each monomial.

\begin{defn}\label{iso0} A hiked polynomial is  \textbf{uniform}
if there is some indeterminate $x_i$ for which, in each of its
monomials, the molecule obtained from hiking $x_i$ is semisimple of
the same matrix degree.
\end{defn}

Our objective in this section is to hike to a uniform polynomial.
First we use \S\ref{nonmixed} to dispose of the easy case where each
hiked monomial has a semisimple molecule (Definition~\ref{mol}).

\begin{defn}\label{iso} A radical element of a complex  molecule is  \textbf{isolated}
if multiplication by any radical element on the left or right is
zero. \end{defn}

\begin{rem} The product of two isolated elements is 0, by
definition.
\end{rem}

\begin{prop}\label{closed1} Any polynomial can be hiked to a uniform polynomial.
\end{prop}
\begin{proof} Multiply $x_i$ by a new indeterminate $x_i'$ and hike that.
We are done unless it yields a radical substitution. Since $J^{t+1}
= 0$, we get an isolated element after at most $t$ hikes.

\end{proof}

\subsection{Removing ambiguities  for nonhomogeneous
polynomials having molecules of the same matrix degree}$ $

We have just reduced to the case where all monomials have molecules
of some $x_i$ of the same matrix degree $\tilde n$, but we still must
contend with the possibility that $x_i$ has different degrees in
different monomials, and then our characteristic coefficient
arguments work differently for the different monomials.

Take the finitely many matrix components $R_i, 1 \le i \le k$, each
of which has degree $\tilde n$ and multiplicity $q_i$. Multiplication by
characteristic coefficients $\alpha _i$ is integral over the
multiplication by ~ $\alpha _i^{q_i}$.

 We introduce a commuting indeterminate $\la_i$ for each of the finitely many characteristic
coefficients $\alpha _i,$ $i \in I$, define $C'$ to be $\hat C[\la_i
: i \in I]$. Defining the action of $\la_i$ on $R_i$ to be the same
as that of $\alpha _i$, we have $(\la_i-\alpha _i)R_i = 0 .$

\begin{defn}\label{sym} Given matrices $a_1, \dots, a_t,$ the \textbf{symmetrized} $(k;j)$ characteristic coefficient is the $j$-elementary symmetric function applied to the $k$-characteristic
coefficients of $a_1, \dots, a_t.$ \end{defn} For example, taking
$k=1$, the symmetrized $(1,j)$-characteristic coefficients $\a_t$
are
$$\sum_{j=1}^t \tr(a_j),\quad \sum_{j_1 >j_2} \tr         (a_{j_1})\tr(a_{j_2}),\quad \dots , \quad \prod_{j=1}^t \tr(a_{j}).$$

\begin{lem}\label{symcoef} Any characteristic coefficient $\a _k$ is integral over the
ring  with all the  symmetrized  characteristic coefficients
adjoined.\end{lem}
\begin{proof} If $\a _{k,t}$ denotes the $(t;j)$-characteristic
coefficient, then  $\a _k$  satisfies the usual polynomial $\la ^ n+
(-1)^j \sum _{j=1}^n \a _{k,j} \la ^{t-j}.$
\end{proof}

\begin{rem}\label{symchar} The reason that we need to introduce the symmetrized characteristic
coefficient is that we might have several glued components and their
molecules, which we cannot distinguish, so we need to find
coefficients common to all of them.
\end{rem}

\begin{prop}\label{q22} Take $\hat f$ of Lemma~\ref{q2}. There is a uniform polynomial $\tilde f $
hiked from $\hat f$ which is characteristic coefficient absorbing.
\end{prop}

\begin{proof}
 $\hat f$ has no semisimple evaluations into a maximal
component, and thus (inductively) all evaluations of $\hat f$
involving a maximal component have a ``radical bridge,'' i.e., $\hat
f$ has a monomial $h$ where   the indeterminates $x_i$ specialize to
external radical substitutions. In particular $n_i =  \tilde n =
n_1.$

We adjoin the characteristic values of all the complex molecules.
Then we apply Proposition~\ref{closed1} to Lemmas~\ref{q1} and
~\ref{q2}.
\end{proof}

\section{Application of Shirshov's theorem}

Here is the connection to full quivers.
\begin{defn}\label{Ahat}
For a \Zcd\ algebra $A\sub \M[n](K)$ faithful over an integral
domain $C$, we denote by $\hat C$ the algebra obtained by adjoining
to $C$ the  symmetrized characteristic coefficients of products of
the Peirce components of the generic generators of~$A$ (of length up
to the bound of Shirshov's Theorem~ \cite[Chapter 2]{BR}).
\end{defn}

 { \bf{Characteristic Value Adjunction Theorem
\cite[Theorem~3.22]{BRV5}}}\label{tradj}.
  Let $\mathcal A_0$ denote the algebra obtained by adjoining to $A_0$
the
  characteristic matrix coefficients of
products of the sub-Peirce components of the generic generators of
$A_0$ (of length up to the bound of Shirshov's Theorem~
\cite[Chapter 2]{BR}), and let $\hat C$ be the algebra obtained by
adjoining to $F$ these symmetrized characteristic coefficients. The
T-ideal~$\CI$ generated by the polynomial~$\tilde f$ contains a
nonzero T-ideal which is also an ideal of the algebra~$\hat A_0$.
Recall in view of Shirshov's theorem that we only need to adjoin a
finite number of elements to obtain $\hat C.$

\begin{lem}\label{Ahatisfinite}
The algebra $\mathcal A_0$ is a finite module over $\hat C$, and in
particular is Noetherian and representable.
\end{lem}
\begin{proof}
Let $\hat C'$ be the commutative algebra generated over $C$, by all
the characteristic coefficients of (finitely many) products of the
 Peirce components of the generic generators of $A_0$, as in
\Dref{Ahat}. Clearly $\hat{C} \sub \hat{C}'$.

Enlarge $\mathcal A_0 $ to $\mathcal A_0' = \hat{C}'\mathcal A_0 $,
which is a finite module over $\hat{C}'$ in view of Shirshov's
Theorem. But $\hat C'$ is finite over $\hat C,$ in view of
Lemma~\ref{symcoef}, implying $\mathcal A_0$ is finite over~$\hat
C$.
 Thus  $\mathcal A_0$  is Noetherian, and is representable by Anan'in's Theorem \cite{An}.
\end{proof}

Also, for any characteristic coefficient-absorbing polynomial $f$
with respect to the quiver of $A_0$, the Hamilton-Cayley identity
induced by $f$ is an identity of~$\mathcal A_0$, and thus of $A_0$.

\section{Conclusion of the proof of
Theorem~\ref{4.66}}\label{CIO}$ $

\begin{proof}
%

  Let  $\CI $ be the T-ideal generated by $\tilde f,$ and $\CI_1$ be the
  T-ideal of $\hat {A_0}$ generated by symmetrized $\bar{q}$-characteristic
coefficient-absorbing polynomials of $\CI$  in $\tilde A: = \hat
C' A_0.$ The ideal $\CI_0$ is representable by
Lemma~\ref{Ahatisfinite}, implying $A_0 \cap \CI_1$ is
representable, as desired.\end{proof}
%


\end{document}